\documentclass{amsproc}
\usepackage{amsmath,amsthm,amscd,amssymb}
\usepackage{geometry}                % See geometry.pdf to learn the layout options. There are lots.
\numberwithin{equation}{section}

\theoremstyle{plain}
\newtheorem{theorem}{Theorem}[section]
\newtheorem{lemma}[theorem]{Lemma}
\newtheorem{corollary}[theorem]{Corollary}

\theoremstyle{definition}
\newtheorem{definition}[theorem]{Definition}

\newtheorem{case[theorem]}{Case}

\theoremstyle{remark}
\newtheorem{remark}[theorem]{Remark}

\numberwithin{equation}{section}

	\newcommand{\abss}[1]{||#1||}
\begin{document}

\title{Cycles of arbitrary length in distance graphs on ${\Bbb F}_q^d$}


\author{A. Iosevich, G. Jardine and B. McDonald}

\date{today}

\address{Department of Mathematics, University of Rochester, Rochester, NY}

\email{iosevich@math.rochester.edu}
\email{gjardine@ur.rochester.edu}
\email{bmcdon11@ur.rochester.edu}

\thanks{The research of the first listed author was partially supported by the National Science Foundation grant no. HDR TRIPODS - 1934962.}

\begin{abstract} For $E \subset {\Bbb F}_q^d$, $d \ge 2$, where ${\Bbb F}_q$ is the finite field with $q$ elements, we consider the distance graph ${\mathcal G}^{dist}_t(E)$, $t \not=0$, where the vertices are the elements of $E$, and two vertices $x$, $y$ are connected by an edge if $||x-y|| \equiv {(x_1-y_1)}^2+\dots+{(x_d-y_d)}^2=t$. We prove that if $|E| \ge C_k q^{\frac{d+2}{2}}$, then ${\mathcal G}^{dist}_t(E)$ contains a statistically correct number of cycles of length $k$. We are also going to consider the dot-product graph ${\mathcal G}^{prod}_t(E)$, $t \not=0$, where the vertices are the elements of $E$, and two vertices $x$, $y$ are connected by an edge if $x \cdot y \equiv x_1y_1+\dots+x_dy_d=t$. We obtain similar results in this case using more sophisticated methods necessitated by the fact that the function $x \cdot y$ is not translation invariant. The exponent $\frac{d+2}{2}$ is improved for sufficiently long cycles. 
\end{abstract}

\maketitle

\tableofcontents

\section{Introduction}

\vskip.125in

Let ${\Bbb F}_q$ denote the finite field with $q$ elements, and let ${\Bbb F}_q^d$ denote the $d$-dimensional vector space over this field. Given $E \subset {\Bbb F}_q^d$, $d \ge 2$, define the distance graph ${\mathcal G}^{dist}_t(E)$, $t \not=0$ by letting the points in $E$ be vertices and connecting two vertices $x$ and $y$ by an edge if 
$$ ||x-y|| \equiv \sum_{i=1}^d {(x_i-y_i)}^2=t.$$

Similarly, define the dot product graph ${\mathcal G}^{prod}_t(E)$, $t \not=0$ by letting the points in $E$ be vertices and connecting two vertices $x$ and $y$ by an edge if 
$$ x \cdot y \equiv \sum_{i=1}^d x_iy_i=t.$$ 

Considerable progress has been achieved in the study of ${\mathcal G}^{dist}_t(E)$ and ${\mathcal G}^{prod}_t(E)$ over the years. See, for example, \cite{BCCHIP}, \cite{BHIPR17}, \cite{BIP12}, \cite{CEHIK12}, \cite{CHIKR}, \cite{HI08}, \cite{HIKR11}, \cite{IP19}, \cite{IR07},  \cite{PPV17}, \cite{V11}, and others. In \cite{IR07}, the first listed author and Rudnev proved that if $E \subset {\Bbb F}_q^d$, $d \ge 2$, and $t \not=0$, then 
\begin{equation} \label{IRidentity} |\{(x,y) \in E \times E: ||x-y||=t \}|={|E|}^2q^{-1}+D(E), \end{equation} where 
\begin{equation} \label{distanceremainderestimate} |D(E)| \leq 2q^{\frac{d-1}{2}}|E|. \end{equation}

In particular, this implies that if $t \not=0$ and $|E|$ is larger than $4q^{\frac{d+1}{2}}$, then the number of edges in ${\mathcal G}^{dist}_t(E)$ is at least $\frac{1}{2}{|E|}^2q^{-1}$, and as the size of $|E|$ increases, the number of edges approaches ${|E|}^2q^{-1}$. 

In \cite{HI08}, Derrick Hart and the first listed author proved that if $t \not=0$ and $E \subset {\Bbb F}_q^d$, then 
\begin{equation} \label{HIidentity} |\{(x,y) \in E \times E: x \cdot y=t \}|={|E|}^2q^{-1}+R(E), \end{equation} where 
\begin{equation} \label{dotproductremainderestimate} |R(E)| \leq q^{\frac{d-1}{2}}|E|. \end{equation} 

Once again, this means that the number of edges in ${\mathcal G}_t^{prod}(E)$ approaches ${|E|}^2q^{-1}$ as $|E|$ increases above $q^{\frac{d+1}{2}}$. 

\begin{definition}
Let $C_n^{dist}$ be the number of cycles of length $n$ in $\mathcal{G}_t^{dist}$, i.e.
\begin{align*}
C_n^{dist}=\left|\{(x_1,x_2,...,x_n)\in E^n: ||x_1-x_2||=||x_2-x_3||=\cdots =||x_n-x_1||\}\right|
\end{align*}

 and similarly let $C_n^{prod}$ be the number of cycles of length $n$ in $\mathcal{G}_t^{prod}$.
\end{definition}
Our main results are to estimate $C_n^{dist}$, $C_n^{prod}$.  Heuristically since there are $|E|^n$ $n$-tuples of vertices in $\mathcal{G}_t^{dist}$ and $\mathcal{G}_t^{prod}$, and a randomly chosen pair of points in $\mathbb{F}_q^d$ will have distance, dot-product, respectively, $t$ with probability $\sim 1/q$, we expect that $C_n^{prod}\sim q^{-n}|E|^n$ when $E$ is large.  This is made precise by the following theorem.
\begin{theorem}\label{main}
Let
\begin{align*}
\gamma=\left\{
\begin{array}{ll}
-1 & :d=2 \\
-\frac{d-2}{2} & :d\geq 3
\end{array}
\right.
\end{align*}
If
\begin{align*}
12q^{\gamma}+8\frac{q^{d+2}}{|E|^2}+\left(24+12\left\lfloor \frac{n}{2}\right\rfloor\right)\frac{q^{\frac{d+1}{2}}}{|E|}\leq 1,
\end{align*}
then for $n\geq 6$,
\begin{align*}
\left|C_n^{dist}-\frac{|E|^n}{q^n}\right|
\leq \frac{|E|^n}{q^n}\left(12q^{\gamma}+8\frac{q^{d+1}}{|E|^2}+\left(24+12\left\lfloor \frac{n}{2}\right\rfloor\right)\frac{q^{\frac{d+1}{2}}}{|E|}\right)
\end{align*}
Moreover,
\begin{align*}
\left|C_4^{dist}-\frac{|E|^4}{q^4}\right|
&\leq \frac{|E|^4}{q^4}\left(12q^{\gamma}+8\frac{q^{d+2}}{|E|^2}+28\frac{q^{\frac{d+1}{2}}}{|E|}\right) \\
\left|C_5^{dist}-\frac{|E|^5}{q^5}\right|
&\leq \frac{|E|^5}{q^5}\left(12q^{\gamma}+8\frac{q^{\frac{2d+3}{2}}}{|E|^2}+32\frac{q^{\frac{d+1}{2}}}{|E|}\right)
\end{align*}
The same is true for $C_n^{prod}$.
\end{theorem}
This theorem says that if $|E|$ is much bigger than $q^{\frac{d+2}{2}}$ for $q$ sufficiently large, then $C_n^{dist},C_n^{prod}$ are very close to $\frac{|E|^n}{q^n}$.  For the case $n=4$, we cannot get a nontrivial result without the restriction on the size of $E$.  However, for larger $n$ we can accept a weaker restriction on the size of $E$, at the cost of a weaker error term.  The techniques for proving these theorems are essentially the same.
\begin{theorem}\label{main2}
For $n\geq 5$ and $q$ sufficiently large,
\begin{align*}
\left|C_n^{dist}-\frac{|E|^n}{q^n}\right|
\leq \left(36+80\cdot 6^{\left \lfloor \frac{n}{2}\right\rfloor-2}+12\left\lfloor \frac{n}{2}\right\rfloor\right)\frac{|E|^n}{q^n}q^{-\left(\frac{n}{2}-1\right)\delta}
\end{align*}
whenever
\begin{align*}
|E|\geq \left\{\begin{array}{ll}
q^{\frac{1}{2}\left(d+2-\frac{k-2}{k-1}+\delta\right)} & : n=2k, \ \text{even} \\
q^{\frac{1}{2}\left(d+2-\frac{2k-3}{2k-1}+\delta\right)} & :n=2k+1 \ \text{odd}
\end{array}
\right.
\end{align*}
where 
\begin{align*}
0<\delta<\frac{1}{2\left\lfloor\frac{n}{2}\right\rfloor^2}
\end{align*}
\end{theorem}
Our final variant of Theorem \ref{main} is to count non-degenerate cycles.
\begin{definition}
Let $\mathcal{N}_n^{dist}$ resp. $\mathcal{N}_n^{prod}$ be the number of non-degenerate cycles in $\mathcal{G}_t^{dist}$, $\mathcal{G}_t^{prod}$, respectively, i.e. cycles $x_1,...,x_n$ with $x_i\neq x_j$ when $i\neq j$.  
\end{definition}
\begin{theorem}\label{nondegenerate}
For $n\geq 4$ and $q$ sufficiently large, if
\begin{align*}
|E|\geq \left\{\begin{array}{ll}
q^{\frac{1}{2}\left(d+2-\frac{k-2}{k-1}+\delta\right)} & :n=2k \\
q^{\frac{1}{2}\left(d+2-\frac{2k-3}{2k-1}+\delta\right)} & :n=2k+1
\end{array}
\right.
\end{align*}
then
\begin{align*}
\left|\mathcal{N}_n^{dist}-\frac{|E|^n}{q^n}\right|\leq \frac{|E|^n}{q^n}\left(K_nq^{-\left(\frac{n}{2}-1\right)\delta}+2nq^{-\frac{2}{n-1}}+c_nq^{-\frac{d-3}{2}-\varepsilon}\right)
\end{align*}
where $K_4=48$, and 
\begin{align*}
K_n&=36+80\cdot 6^{\left\lfloor\frac{n}{2}\right\rfloor-2}+12\left\lfloor\frac{n}{2}\right\rfloor \ \ \text{for} \ n\geq 5 \\
c_n&=(n-1)^{n-3}\cdot 2^{\binom{n-1}{2}-n+3} \\
\varepsilon&=\left\{\begin{array}{ll}
1-\frac{k-2}{k-1}+\delta & : n=2k \\
1-\frac{2k-3}{2k-1}+\delta & :n=2k+1
\end{array}
\right.
\end{align*}
The same estimates hold for $\mathcal{N}_n^{prod}$.
\end{theorem}

\vskip.125in 

\begin{remark} A close variant of Theorem \ref{nondegenerate} was proved in \cite{IP19} in the context of the distance graph with a weaker exponent, namely $\frac{d+3}{2}$ in place of $\frac{d+2}{2}$ above. Moreover, we get a better exponent than $\frac{d+2}{2}$ above for long cycles. The main result in \cite{IP19} handles more general configurations and this raises the question of whether the techniques of this paper can be used obtain results for general configurations. We shall address this issue in the sequel. \end{remark} 

\vskip.125in 

The proof of Theorem \ref{main} requires an estimate for the number of paths of a given length in $\mathcal{G}_t^{dist}$, $\mathcal{G}_t^{prod}$. This has been done for distances in \cite{BCCHIP}, and we will cover the dot-product case in section \ref{paths}.  See also \cite{Sht18}, where an improvement on the length of the possible paths is obtained. To count non-degenerate cycles we will also need to consider the number of embeddings of an arbitrary tree $T$ into $\mathcal{G}_t^{dist}$,$\mathcal{G}_t^{prod}$, which is handled in section \ref{degenerate}.  The following result is sufficient to show existence of non-degenerate cycles for $d\geq 3$ under the hypotheses of Theorem \ref{nondegenerate}
\begin{theorem}\label{tree counting}
Fix a tree $T$ with $r+1$ vertices and hence $r$ edges.  For $\epsilon>0$, if $|E|>q^{\frac{d+1}{2}+\varepsilon}$ then there is a subset $E'\subseteq E$ with 
\begin{align*}
|E\setminus E'|\leq 2q^{-\frac{2\varepsilon}{r+1}}|E|,
\end{align*}
and if $n_T$ is the number of embeddings of $T$ into $\mathcal{G}_t^{dist}(E')$, then
\begin{align*}
\left|n_T-\frac{|E|^{r+1}}{q^r}\right|\leq 8\frac{|E|^{r+1}}{q^{r}}q^{-\frac{2\epsilon}{r+1}}
\end{align*}
The same is true replacing $\mathcal{G}_t^{dist}$ with $\mathcal{G}_t^{prod}$. 
\end{theorem}

\vskip.125in 

We note that in this paper, we obtain the same results for the distance graph and the dot product graph. While the techniques are, at least superficially, somewhat different due to the lack of translation invariance in the dot product setting, it is reasonable to ask whether a general formalism is possible. The underlying edge operator in the distance graph is 
$$ A_tf(x)=\sum_{||x-y||=t} f(y),$$ while the edge operator in the dot product case is 
$$ R_tf(x)=\sum_{x \cdot y=t} f(y) dy.$$ 

The Euclidean variant of $Af(x)$ is 
$$ {\mathcal A}_tf(x)=\int f(x-y) d\sigma(y),$$ where $\sigma$ the surface measure on $S^{d-1}$, the unit sphere. The Euclidean variant of $Rf(x)$ is the classical Radon transform 
$$ {\mathcal R}_tf(x)=\int_{x \cdot y=t} \psi(y) f(y) d\sigma_{x,t}(y),$$ where $\psi$ is a smooth cut-off function and $\sigma_{x.t}$ is the surface measure on $\{y \in {\Bbb R}^d: x \cdot y=t \}$. If $t \not=0$, both ${\mathcal A}_t$ and ${\mathcal R}_t$ map $L^2({\Bbb R}^d)$ to $H^{\frac{d-1}{2}}({\Bbb R}^d)$, where $H^s({\Bbb R}^d)$ is the Sobolev space of $L^2({\Bbb R}^d)$ function with generalized derivative of order $s>0$ in $L^2({\Bbb R}^d$. See, for example, \cite{So93} and the references contained therein. 

A significant amount of progress has been made in the Euclidean setting in studying general configuration problems from the point of view of Sobolev estimate. See, for example, \cite{GIT19}, \cite{GIT20} and \cite{ITU18} for some recent work in this direction. It would be interesting to encode the bounds in the finite field setting using a suitable formalism analogous to their Euclidean counterparts. Both the edge operator $A_t$ and $R_t$, defined above, satisfy the following bounds that we encode as follows. Let 
$$ T^{\phi}_tf(x)=\sum_{\phi(x,y)=t} f(y), $$ where $\phi: {\Bbb F}_q^d \times {\Bbb F}_q^d \to {\Bbb F}_q$, a function. See \cite{PS91} for the description of the continuous analog of this family of operators, introduced by Phong and Stein. 

Note that if $\phi(x,y)=||x-y||$, we recover the operator $A_t$, while if $\phi(\delta,y)=x \cdot y$, we cover $R_t$. Let 
$$ T^{\phi}_{t,0}f(x)=T^{\phi}_tf(x)-q^{-d} \sum_{x \in {\Bbb F}_q^d} T^{\phi}_tf(x),$$ which amounts to stripping $Tf(x)$ of its $0$'th Fourier coefficient. Both $A_t$ and $R_t$ above satisfy the bound 

\begin{equation} \label{finitesobolev} <T^{\phi}_{t,0}f,g> \ \leq \ Cq^{\frac{d-1}{2}} {||f||}_{L^2({\Bbb F}_q^d)} \cdot {||g||}_{L^2({\Bbb F}_q^d)}, \end{equation} where 
$$ {||f||}_{L^2({\Bbb F}_q^d)}^2=\sum_{x \in {\Bbb F}_q^d} {|f(x)|}^2,$$ and the inner product on the left hand side above is the $L^2({\Bbb F}_q^d)$ inner product. 

The estimate (\ref{finitesobolev}) can be viewed as an analog of the $L^2({\Bbb R}^d) \to H^{\frac{d-1}{2}}({\Bbb R}^d)$ bound in the Euclidean case since for any $\phi$ satisfying 
\begin{equation} \label{sizecondition} |\{x \in {\Bbb F}_q^d: \phi(x,y)=t\}|= |\{y \in {\Bbb F}_q^d: \phi(x,y)=t\}| \approx q^{d-1}, \end{equation} the estimate (\ref{finitesobolev}) holds with $q^{\frac{d-1}{2}}$ replaced by $q^{d-1}$. It is reasonable to summarize the above using the following notion. 

\begin{definition} Let $T^{\phi}_t$, $T^{\phi}_{t,0}$ be as above. Suppose that  (\ref{sizecondition}) holds and, in place of (\ref{finitesobolev}) we have 
\begin{equation} \label{finitesobolevgeneral} <T^{\phi}_{t,0}f,g> \ \leq \ Cq^{d-1} q^{-\alpha} {||f||}_{L^2({\Bbb F}_q^d)} \cdot {||g||}_{L^2({\Bbb F}_q^d)} \end{equation} for some $\alpha>0$. 

Then we say that $T^{\phi}_{t,0}$ is smoothing of order $\alpha$. 
\end{definition} 

Given a function $\phi: {\Bbb F}_q^d \times {\Bbb F}_q^d$, as above, and $E \subset {\Bbb F}_q^d$, define ${\mathcal G}_t^{\phi}(E)$ to be the graph where the vertices are given by the 
points of $E$, and two vertices are connected by an edge of $\phi(x,y)=t$. In the sequel, we shall engage in a systematic study of the properties of this graph under the smoothing assumption (\ref{finitesobolevgeneral}) and the size assumption (\ref{sizecondition}) above. 

\vskip.125in 

\subsection{Acknowledgements} This paper is dedicated to Professor Vinogradov's 130th birthday. The authors wish to make it clear that they are celebrating Vinogradov's mathematical legacy. In particular, this submission should not be viewed as an endorsement, in any way, of Vinogradov's political and social views. 

\section{Paths on the Dot Product Graph of $E$}\label{paths}

\vskip.125in

\begin{definition}
Let $\mathcal{P}_k$ be the number of paths of length $k$ on $\mathcal{G}_t^{prod}(E)$.
\end{definition}

\begin{theorem}\label{Chains}
If $|E|>\frac{k}{\log{2}}q^{\frac{d+1}{2}}$, then
\begin{align*}
\left|\mathcal{P}_k-\frac{|E|^{k+1}}{q^k}\right|\leq \frac{k}{\log{2}}q^{\frac{d+1}{2}}\frac{|E|^k}{q^k}
\end{align*}
\end{theorem}

We present an argument similar to the proof of Theorem 1.1 in \cite{BCCHIP}.  First we state Theorem 2.1 from \cite{CHIKR}:
\begin{theorem}\label{CHIKR functional}
For non-negative functions $f,g$ on $\mathbb{F}_q^d$,
\begin{align*}
\sum_{x\cdot y=t}{f(x)g(y)}=q^{-1}||f||_{1}||g||_{1}+R(t)
\end{align*}
where
\begin{align*}
|R(t)|\leq ||f||_{2}||g||_{2}q^{\frac{d-1}{2}}
\end{align*}
if $t\neq 0$.
\end{theorem}
Note that when $f,g$ are both the indicator function of $E$, this reduces to the previously mentioned result from \cite{HI08}.  Now we will obtain recursive estimates for $\mathcal{P}_k$, which are slightly different in the even and odd cases.  The induction works by decomposing a path into two paths of about half the length, joined together by an additional edge. 
\begin{lemma}
\begin{align*}
\mathcal{P}_{2k+1}&=q^{-1}\mathcal{P}_k^2+R_{2k+1}, \nonumber \\
\mathcal{P}_{2k}&=q^{-1}\mathcal{P}_k\mathcal{P}_{k-1}+R_{2k}
\end{align*}
where
\begin{align*}
|R_{2k+1}|&\leq q^{\frac{d-1}{2}}\mathcal{P}_{2k}, \nonumber \\
|R_{2k}|&\leq q^{\frac{d-1}{2}}\sqrt{\mathcal{P}_{2k}\mathcal{P}_{2k-2}}
\end{align*}
\end{lemma}
\begin{proof}
Let 
\begin{align*}
f_1(x)&=\sum_y{E(x)E(y)D_t(x,y)}, \nonumber \\
f_{k+1}(x)&=\sum_y{E(x)f_k(y)D_t(x,y)},
\end{align*}
where $D_t(x,y)=1$ when $x\cdot y=t$, and is 0 otherwise.  Note that $f_k(x)$ is the number of paths of length $k$ in $\mathcal{G}_t^{prod}$ starting at $x$.  Then by Theorem \ref{CHIKR functional},
\begin{align*}
\mathcal{P}_{2k+1}&=\sum_{x,y}{f_k(x)f_k(y)D_t(x,y)}
=q^{-1}\left(\sum_{x}{f_k(x)}\right)^2+R_{2k+1} \nonumber \\
&=q^{-1}\mathcal{P}_k^2+R_{2k+1}
\end{align*}
with
\begin{align*}
|R_{2k+1}|\leq q^{\frac{d-1}{2}}\left(\sum_x{f_k(x)^2}\right)=q^{\frac{d-1}{2}}\mathcal{P}_{2k}
\end{align*}
Similarly, we have
\begin{align*}
\mathcal{P}_{2k}&=\sum_{x,y}{f_k(x)f_{k-1}(y)D_t(x,y)}
=q^{-1}\left(\sum_{x}{f_k(x)}\right)\left(\sum_y{f_{k-1}(y)}\right)+R_{2k} \nonumber \\
&=q^{-1}\mathcal{P}_k\mathcal{P}_{k-1}+R_{2k}
\end{align*}
with
\begin{align*}
|R_{2k}|\leq q^{\frac{d-1}{2}}||f_k||_{2}||f_{k-1}||_{2}
=q^{\frac{d-1}{2}}\sqrt{\mathcal{P}_{2k}\mathcal{P}_{2k-2}}
\end{align*}
\end{proof}
Since the bound for the error term $R_k$ still may depend on $\mathcal{P}_k,\mathcal{P}_{k-1},\mathcal{P}_{k-2}$, we will need an upper bound on $\mathcal{P}_k$ for this to be useful.
\begin{lemma}
Let
\begin{align*}
X=\frac{|E|+q^{\frac{d+1}{2}}}{q}
\end{align*}
Then 
\begin{align*}
\mathcal{P}_k\leq |E|X^k
\end{align*}
\end{lemma}
\begin{proof}
We proceed by induction, noting that the case $k=1$ follows from Theorem \ref{CHIKR functional}.  For the odd case, we get
\begin{align*}
\mathcal{P}_{2k+1}&\leq q^{-1}\mathcal{P}_k^2+q^{\frac{d-1}{2}}\mathcal{P}_{2k}
\leq q^{-1}|E|^2X^{2k}+q^{\frac{d-1}{2}}|E|X^{2k} \nonumber \\
&=|E|X^{2k+1}
\end{align*}
For the even case,
\begin{align*}
|\mathcal{P}_{2k}-q^{-1}\mathcal{P}_k\mathcal{P}_{k-1}|\leq q^{\frac{d-1}{2}}\sqrt{\mathcal{P}_{2k}\mathcal{P}_{2k-2}}
\end{align*}
Squaring and then completing the square to solve for $\mathcal{P}_{2k}$, we find that
\begin{align*}
\left(\mathcal{P}_{2k}-\frac{1}{2}q^{d-1}\mathcal{P}_{2k-2}-q^{-1}\mathcal{P}_k\mathcal{P}_{k-1}\right)^2
\leq \frac{1}{4}q^{2d-2}\mathcal{P}_{2k-2}^2+q^{d-2}\mathcal{P}_k\mathcal{P}_{k-1}\mathcal{P}_{2k-2}
\end{align*}
And thus
\begin{align*}
\mathcal{P}_{2k}&\leq q^{-1}\mathcal{P}_k\mathcal{P}_{k-1}+\frac{1}{2}q^{d-1}\mathcal{P}_{2k-2}+\sqrt{\frac{1}{4}q^{2d-2}\mathcal{P}_{2k-2}^2+q^{d-2}\mathcal{P}_k\mathcal{P}_{k-1}\mathcal{P}_{2k-2}} \nonumber \\
&\leq \frac{|E|^2}{q}X^{2k-1}+\frac{1}{2}q^{d-1}|E|X^{2k-2}+\sqrt{\frac{1}{4}q^{2d-2}|E|^2X^{4k-4}+q^{d-2}|E|^3X^{4k-3}} \nonumber \\
&=\frac{|E|^2}{q}X^{2k-1}+\frac{1}{2}q^{d-1}|E|X^{2k-2}\left(1+\sqrt{1+4q^{-d}|E|X}\right)
\end{align*}
Now,
\begin{align*}
1+4q^{-d}|E|X=\frac{q^{d+1}+4|E|^2+4q^{\frac{d+1}{2}}|E|}{q^{d+1}}
=\frac{\left(q^{\frac{d+1}{2}}+2|E|\right)^2}{q^{d+1}}
\end{align*}
so
\begin{align*}
\mathcal{P}_{2k}&\leq \frac{|E|^2}{q}X^{2k-1}+q^{d-1}|E|X^{2k-2}\left(1+\frac{|E|}{q^{\frac{d+1}{2}}}\right) \nonumber \\
&=|E|X^{2k-2}\left(\frac{|E|X}{q}+q^{d-1}+q^{\frac{d-3}{2}}|E|\right) \nonumber \\
&=|E|X^{2k-2}\left(\frac{|E|^2+q^{\frac{d+1}{2}}|E|}{q^2}+q^{d-1}+q^{\frac{d-3}{2}}|E|\right) \nonumber \\
&=|E|X^{2k-2}\left(\frac{|E|^2+2q^{\frac{d+1}{2}}|E|+q^{d+1}}{q^2}\right) \nonumber \\
&=|E|X^{2k}
\end{align*}
\end{proof}
We now have enough control over the error term $R_k$ to prove the Theorem.
\begin{proof}[Proof of Theorem \ref{Chains}]
The binomial expansion for $X^k$ yields
\begin{align*}
\mathcal{P}_k&\leq |E|X^k
=\frac{E}{q^k}\sum_{i=0}^k{\binom{k}{i}|E|^{k-i}q^{i\frac{d+1}{2}}} \nonumber \\
&=\frac{|E|^{k+1}}{q^k}+\frac{q^{\frac{d+1}{2}}|E|}{q^k}\sum_{i=1}^k{\binom{k}{i}|E|^{k-i}q^{(i-1)\frac{d+1}{2}}} \nonumber \\
&\leq \frac{|E|^{k+1}}{q^k}+\frac{q^{\frac{d+1}{2}}|E|^k}{q^k}\sum_{i=1}^k{\binom{k}{i}\left(\frac{\log{2}}{k}\right)^{i-1}} \nonumber \\
&=\frac{|E|^{k+1}}{q^k}+\frac{k}{\log{2}}q^{\frac{d+1}{2}}\frac{|E|^k}{q^k}\left(\left(1+\frac{\log{2}}{k}\right)^k-1\right) \nonumber \\
&\leq \frac{|E|^{k+1}}{q^k}+\frac{k}{\log{2}}q^{\frac{d+1}{2}}\frac{|E|^k}{q^k}
\end{align*}
We will show that $\mathcal{P}_k\geq \frac{|E|^{k+1}}{q^k}-\frac{k}{\log{2}}q^{\frac{d-1}{2}}|E|$ by induction, noting that the case $k=1$ follows from Theorem \ref{CHIKR functional}.  Suppose that the result holds for indices less than $2k+1$.  Then
\begin{align*}
\mathcal{P}_{2k+1}&\geq \frac{\mathcal{P}_k^2}{q}-q^{\frac{d-1}{2}}\mathcal{P}_{2k} \nonumber \\
&\geq q^{-1}\left(\frac{|E|^{k+1}}{q^k}-\frac{k}{\log{2}}q^{\frac{d+1}{2}}\frac{|E|^k}{q^k}\right)^2-q^{\frac{d-1}{2}}\left(\frac{|E|^{2k+1}}{q^{2k}}+\frac{k}{\log{2}}q^{\frac{d+1}{2}}\frac{|E|^{2k}}{q^{2k}}\right) \nonumber \\
&=\frac{|E|^{2k+2}}{q^{2k+1}}-q^{\frac{d+1}{2}}\frac{|E|^{2k+1}}{q^{2k+1}}\left(\frac{2k}{\log{2}}-q^{\frac{d+1}{2}}\frac{k^2}{|E|(\log{2})^2}+1+\frac{k}{|E|\log{2}}q^{\frac{d+1}{2}}\right) \nonumber \\
&\geq \frac{|E|^{2k+2}}{q^{2k+1}}-q^{\frac{d+1}{2}}\frac{|E|^{2k+1}}{q^{2k+1}}\left(\frac{2k}{\log{2}}+1\right) \nonumber \\
&\geq \frac{|E|^{2k+2}}{q^{2k+1}}-\frac{2k+1}{\log{2}}q^{\frac{d+1}{2}}\frac{|E|^{2k+1}}{q^{2k+1}}
\end{align*}
For the even case suppose that the result holds for indices less than $2k$.  Note that we have already established the upper bound for all $k$, and in fact for sufficiently large $q$, we can use the weaker upper bound $\mathcal{P}_{k}\leq \frac{1}{\log{2}}\frac{|E|^{k+1}}{q^k}$ to obtain
\begin{align*}
\mathcal{P}_{2k}&\geq \frac{\mathcal{P}_k\mathcal{P}_{k-1}}{q}-q^{\frac{d-1}{2}}\sqrt{\mathcal{P}_{2k}\mathcal{P}_{2k-2}} \nonumber \\
&\geq q^{-1}\left(\frac{|E|^{k+1}}{q^k}-\frac{k}{\log{2}}q^{\frac{d+1}{2}}\frac{|E|^k}{q^k}\right)\left(\frac{|E|^k}{q^{k-1}}-\frac{k-1}{\log{2}}q^{\frac{d+1}{2}}\frac{|E|^{k-1}}{q^{k-1}}\right)-\frac{1}{\log{2}}q^{\frac{d-1}{2}}\frac{|E|^{2k}}{q^{2k-1}} \nonumber \\
&=\frac{|E|^{2k+1}}{q^{2k}}-q^{\frac{d+1}{2}}\frac{|E|^{2k}}{q^{2k}}\left(\frac{2k-1}{\log{2}}-q^{\frac{d+1}{2}}\frac{k(k-1)}{|E|(\log{2})^2}+\frac{1}{\log{2}}\right) \nonumber \\
&\geq\frac{|E|^{2k+1}}{q^{2k}}-\frac{2k}{\log{2}}q^{\frac{d+1}{2}}\frac{|E|^{2k}}{q^{2k}}
\end{align*}
\end{proof}

\section{Proof of the Main Theorem}

\vskip.125in
%BEGIN insertion %SPHERE
We first present the following functional theorems for the distance and dot-product case respectively, which will allow us to prove the main theorem inductively.

\begin{theorem}\label{distance cycles theorem}
For nonnegative functions $f,g:\mathbb{F}_q^d\times\mathbb{F}_q^d\to \mathbb{R}$, let $F(x)=\sum_y{f(x,y)}$, $G(z)=\sum_w{g(z,w)}$, $F'(y)=\sum_x{f(x,y)}$, $G'(w)=\sum_z{g(z,w)}$.  Then for nonzero $t\in \mathbb{F}_q$,
\begin{align*}
&\left|\sum_{||x-z||=||w-y||=t}{f(x,y)g(z,w)}-q^{-2}||f||_{1}||g||_{1}\right| \\
& \ \ \ \ \ \ \leq 3q^{-\frac{d+2}{2}}||f||_1||g||_1 + 4q^{d-1}||f||_{2}||g||_{2}+
4q^{\frac{d-3}{2}}\left(||F||_{2}||G||_{2}+||F'||_{2}||G'||_{2}\right)
\end{align*}
In the case $d=2$, one has
\begin{align*}
&\left|\sum_{||x-z||=||w-y||=t}{f(x,y)g(z,w)}-q^{-2}||f||_{1}||g||_{1}\right| \\
& \ \ \ \ \ \ \leq 3q^{-3}||f||_1||g||_1 + 4q||f||_{2}||g||_{2}+
4q^{-\frac{1}{2}}\left(||F||_{2}||G||_{2}+||F'||_{2}||G'||_{2}\right)
\end{align*}
\end{theorem}

\vskip.125in

\begin{theorem}\label{functional}
With the same notation as in Theorem \ref{distance cycles theorem},
\begin{align*}
\left|\sum_{x\cdot z=y\cdot w=t}{f(x,y)g(z,w)}-q^{-2}||f||_{1}||g||_{1}\right|
\leq 2q^{d-1}||f||_{2}||g||_{2}+q^{\frac{d-3}{2}}\left(||F||_{2}||G||_{2}+||F'||_{2}||G'||_{2}\right)
\end{align*}
\end{theorem}

\begin{definition}
For $t\in \mathbb{F}_q$, let $S_t=\{x\in \mathbb{F}_q^d: ||x||=t\}$.  We will also identify a set with its indicator function, so that $S_t(x)=1$ when $x\in S_t$.  
\end{definition}

We will use a few facts about the discrete sphere $S_t$.  See for example the appendix of \cite{BHIPR17} for a treatment of a theorem proved by Minkowski at the age of 17 (see \cite{Min11}), of which the following lemma is a special case. 
%Note BCCHIP refers to Iosevich Rudnev, so it might be more polite just to refer to Iosevich Rudnev for these results as well.

\begin{lemma}\label{sphere bounds nonzero}
When $t\neq 0$,

\begin{align*}
|S_t| &= q^{d-1}+\mathcal{E}, \ \text{with} \\
|\mathcal{E}|&\leq q^{\frac{d}{2}}
\end{align*}
Moreover, in the case $d=2$,
\begin{align*}
|S_t|&=q\pm 1
\end{align*}
\end{lemma}
\begin{lemma}
 If $t\neq 0$ and $m\neq 0$, then
\begin{align*}
|\hat{S}_t(m)|\leq 2q^{-\frac{d+1}{2}}
\end{align*}
\end{lemma}

\begin{corollary}
\begin{align*}
|S_t|\leq 2q^{d-1}, \ \text{and}
\end{align*}
\begin{align*}
|S_t|^2&=q^{2d-2}+\mathcal{E}', \text{with} \\
|\mathcal{E}'|& = |2q^{d-1}\mathcal{E}+\mathcal{E}^2| \leq 3q^{\frac{3d-2}{2}}
\end{align*}
In the case $d=2$, 
\begin{align*}
|\mathcal{E}'|\leq 3q
\end{align*}
\end{corollary}

%END insertion %SPHERE

\begin{proof}[Proof of Theorem \ref{distance cycles theorem}]

To estimate the relevant sum, by Fourier inversion on $S_t$ and then unpacking the definition of $\hat{f}$ and $\hat{g}$ we find that

\begin{align*}
\sum_{||x-z||=||w-y||=t}{f(x,y)g(z,w)}
&= \sum_{x,y,z,w}{f(x,y)g(z,w)S_t(x-z)S_t(y-w)} \\
&=\sum_{x,y,z,w}\sum_{m,\ell}{\hat{S}_t(m)\hat{S}_t(\ell)\chi((x-z)\cdot m)\chi((y-w)\cdot \ell)f(x,y)g(z,w)} \\
&=q^{4d}\sum_{m,\ell}{\hat{S}_t(m)\hat{S}_t(\ell)\overline{\hat{f}(m,\ell)}\hat{g}(m,\ell)} \\
&=\left(\frac{|S_t|}{q^d}\right)^2||f||_1||g||_1 + R_A+R_B+R_C \\
&=q^{-2}||f||_1||g||_1+q^{-2d}\mathcal{E}'||f||_1||g||_1+R_A+R_B+R_C
\end{align*}
Where $R_A$ corresponds to the terms where $m,\ell\neq 0$, $R_B$ corresponds to the terms where $m\neq 0$ and $\ell=0$, and $R_C$ corresponds to the terms where $m=0$ and $\ell\neq 0$.  Now Theorem \ref{distance cycles theorem} is an immediate corollary of the following lemma:

\end{proof}

\begin{lemma}\label{distance cycles functional lemma}

\begin{align*}
|R_A|&\leq
4q^{d-1}
\abss{f}_2
\abss{g}_2, \\
|R_B|&\leq
4q^{\frac{d-3}{2}}
\abss{F}_2
\abss{G}_2,\\
|R_C|&\leq
4q^{\frac{d-3}{2}}
\abss{F'}_2
\abss{G'}_2.
\end{align*}
\end{lemma}

\begin{proof}
For $R_A$, applying the bound $|\hat{S}_t(m)|\leq 2q^{-\frac{d+1}{2}}$ for $m\neq 0$ and Cauchy-Schwarz, we get
\begin{align*}
|R_A| &= \left|q^{4d}\sum_{m,\ell\neq 0}
{\hat{S}_t(m)\hat{S}_t(\ell)\overline{\hat{f}(m,\ell)}\hat{g}(m,\ell)}\right| \\
&\leq 4q^{3d-1}\sum_{m,\ell\neq 0}{|\hat{f}(m,\ell)|\cdot |\hat{g}(m,\ell)| } \\
&\leq 4q^{3d-1}||\hat{f}||_2\cdot ||\hat{g}||_2
= 4q^{d-1}||f||_2\cdot ||g||_2
\end{align*}
%\section{one zero term: $RB_{f,g}$ and $RC_{f,g}$}
For $R_B$, observe that
\begin{align*}
\hat{f}(m,0) & = 
	q^{-2d}\sum_{x,y\in\mathbb{F}_q^d}\chi(-x\cdot m)f(x,y) \\
	&=q^{-2d}\sum_x{\chi(-m\cdot x)F(x)}=q^{-d}\hat{F}(m)
\end{align*}
and therefore
\begin{align*}
|R_B| &= \left|q^{4d}\sum_{m\neq 0}\hat{S}_t(m)\hat{S}_t(0)\overline{\hat{f}(m,0)}g(m,0)\right| \\
&\leq  q^{3d}|S_t| \cdot \max_m|\hat{S}_t(m)|\sum_{m\neq 0}|\hat{f}(m,0)|\cdot |\hat{g}(m,0)| \\
&=q^{d}|S_t| \cdot \max_m|\hat{S}_t(m)|\sum_{m\neq 0}|\hat{F}(m)|\cdot |\hat{G}(m)| \\
&\leq 4q^{\frac{3d-3}{2}}||\hat{F}||_2\cdot ||\hat{G}||_2 =4q^{\frac{d-3}{2}}||F||_2\cdot ||G||_2
\end{align*}

The bound for $R_C$ follows with the same argument by symmetry. 
\end{proof}

\begin{proof}[Proof of Theorem \ref{functional}]
\begin{align*}
\sum_{x\cdot z=y\cdot w =t}{f(x,y)g(z,w)}
&=q^{-2}\sum_{\substack{x,y,z,w \\ \alpha,\beta\in\mathbb{F}_q}}{f(x,y)g(z,w)\chi(\alpha(x\cdot z-t))\chi(\beta(y\cdot w-t))} \nonumber \\
&=q^{-2}||f||_{1}||g||_{1}+R_D+R_E+R_F
\end{align*}
Where $R_D$ corresponds to the $\alpha,\beta\neq 0$ terms, $R_E$ corresponds to $\alpha\neq 0$, $\beta=0$, and $R_F$ corresponds to $\alpha=0$, $\beta\neq 0$.  By Cauchy-Schwarz,
\begin{align*}
R_D^2&=\left(q^{-2}\sum_{\substack{x,y,z,w \\ \alpha,\beta\neq 0}}{f(x,y)g(z,w)\chi(\alpha(x\cdot z-t))\chi(\beta(y\cdot w-t))}\right)^2 \nonumber \\
&\leq q^{-4}\left(\sum_{x,y}{f(x,y)^2}\right)\cdot \mathcal{S},
\end{align*}
where
\begin{align*}
\mathcal{S}&=\sum_{\substack{x,y,z,z',w,w' \\ \alpha,\alpha',\beta,\beta'\neq 0}}{\chi(\alpha(x\cdot z-t))\chi(-\alpha'(x\cdot z'-t))\chi(\beta(y\cdot w-t))\chi(-\beta'(y'\cdot w'-t))g(z,w)g(z',w')} \nonumber \\
&=\sum_{\substack{x,y,z,z',w,w' \\ \alpha,\alpha',\beta,\beta'}}{g(z,w)g(z',w')\chi(x\cdot (\alpha z-\alpha'z'))\chi(y\cdot (\beta w-\beta'w'))\chi(t(\alpha'-\alpha))\chi(t(\beta'-\beta))} \nonumber \\
&=q^{2d}\left(\sum_{\substack{z,z',w,w' \\ \alpha,\alpha',\beta,\beta'\neq 0 \\ \alpha z=\alpha'z', \ \beta w=\beta'w'}}{g(z,w)g(z',w')\chi(t(\alpha'-\alpha))\chi(t(\beta'-\beta))}\right) \nonumber \\
&=q^{2d}(I+II+III+IV) \nonumber \\
&\lesssim q^{2d}(I+II+III),
\end{align*}
where $I$ corresponds to the terms in the sum where $\alpha=\alpha'$, $\beta=\beta'$, $II$ corresponds to $\alpha\neq \alpha'$, $\beta\neq \beta'$, $III$ corresponds to $\alpha\neq \alpha'$, $\beta=\beta$, and $IV$ corresponds to $\alpha=\alpha'$, $\beta\neq \beta'$.  By symmetry, $III=IV$.  
\begin{align*}
I=\sum_{\substack{\alpha,\beta\neq 0 \\ z,w}}{g(z,w)^2}=q^2\sum_{z,w}{g(z,w)^2}
\end{align*}
For $II$ we make the substitution $a=\alpha/\alpha'$, $=\alpha'$, $u=\beta/\beta'$, $v=\beta'$, and applying Cauchy-Schwarz yields
\begin{align*}
II&=\sum_{\substack{z,w \\ a,u\neq 0,1}}\sum_{b,v\neq 0}{g(z,w)g(az,uw)\chi(tb(1-a))\chi(tv(1-u))} \nonumber \\
&=\sum_{\substack{z,w \\ a,u\neq 0,1}}{g(z,w)g(az,uw)} 
\leq \sum_{a,u\neq 0,1}\left(\sum_{z,w}{g(z,w)^2}\right)^{\frac{1}{2}}\left(\sum_{z,w}{g(az,uw)^2}\right)^{\frac{1}{2}} \nonumber \\
&=\sum_{a,u\neq 0,1}\sum_{z,w}{g(z,w)^2} \leq q^2\sum_{z,w}{g(z,w)^2}.
\end{align*}
We use the same subsitution for $III$, noting that $\beta=\beta'\neq 0$ forces $u=1$ and $v=\beta$.  
\begin{align*}
|III|&=\left|\sum_{\substack{z,w \\ a\neq 0,1 \\ b\neq 0}}\sum_{\beta\neq 0}{g(z,w)g(az,w)\chi(tb(1-a))}\right| \nonumber \\
&=q\left|\sum_{\substack{z,w \\ a\neq 0,1 \\ b\neq 0}}{g(z,w)g(az,w)\chi(tb(1-a))}\right| \nonumber \\
&=q\sum_{\substack{z,w \\ a\neq 0,1}}{g(z,w)g(az,w)} \nonumber \\
&\leq q^2\sum_{z,w}{g(z,w)^2}
\end{align*}
Putting this all together, we have
\begin{align*}
R_D\leq 2q^{d-1}||f||_{2}||g||_{2}
\end{align*}
Now to estimate $R_E$, which corresponds to the terms $\alpha\neq 0$, $\beta=0$,
\begin{align*}
R_E+q^{-2}||f||_{1}||g||_{1}
&= q^{-2}\sum_{x,y,z,w}\sum_{\alpha\in \mathbb{F}_q}{\chi(\alpha(x\cdot z-t))f(x,y)g(z,w)} \nonumber \\
&=q^{-1}\sum_{x,y,z,w}{f(x,y)g(z,w)D_t(x,z)} \nonumber \\
&=q^{-1}\sum_{x,z}{F(x)G(y)D_t(x,z)} \nonumber \\
\end{align*}
So by Theorem \ref{CHIKR functional}, 
\begin{align*}
|R_E|\leq q^{\frac{d-3}{2}}||F||_{2}||G||_{2}
\end{align*}
By symmetry, the same argument shows that
\begin{align*}
|R_F|\leq q^{\frac{d-3}{2}}||F'||_{2}||G'||_{2}
\end{align*}
\end{proof}
With Theorem \ref{distance cycles theorem} and Theorem \ref{functional}, the proof of Theorem \ref{main} is now a straightforward calculation.  In fact, for the rest of the argument the distance and dot-product cases will be unified.  For brevity we will use the following notation.
\begin{definition}
Let $P_k^{dist}(x,y)$ resp. $P_k^{prod}(x,y)$ be the number of paths of length $k$ from $x\in E$ to $y\in E$ in $\mathcal{G}_t^{dist}$ resp. $\mathcal{G}_t^{prod}$.  We will use $P_k^*(x,y)$ and $\mathcal{G}_t^*$ to refer simultaneously to the distance and dot-product versions; $P_k^*(x,y)$ may essentially be read as $P_k^{dist}(x,y)$ resp. $P_k^{prod}(x,y)$.  
\end{definition}
\begin{definition}
Furthermore, let $\mathcal{P}_k^*=\sum_{x,y}{P_k^*(x,y)}$, the number of paths of length $k$ from $x$ to $y$ on $\mathcal{G}_t^*$.  Similarly let $C_k^*$ be the number of cycles of length $k$ on $\mathcal{G}_t^*$.
\end{definition}
\begin{definition}
Let 
\begin{align*}
\mathcal{T}_t^{dist}(f,g)
=\sum_{||x-z||=||y-w||=t}{f(x,y)g(z,w)},
\end{align*}
\begin{align*}
\mathcal{T}_t^{prod}(f,g)=\sum_{x\cdot z=y\cdot w=t}{f(x,y)g(z,w)}
\end{align*}
and similarly to before $\mathcal{T}_t^*$ will refer to $\mathcal{T}_t^{dist}$ resp. $\mathcal{T}_t^{prod}$.
\end{definition}
Now Theorem \ref{distance cycles theorem} and Theorem \ref{functional} can be stated together concisely:
\begin{corollary}\label{concise}
\begin{align*}
\left|\mathcal{T}_t^*(f,g)-\frac{1}{q^2}||f||_1||g||_1\right|
\leq 3q^{-\frac{d+2}{2}}||f||_1||g||_1+4q^{d-1}||f||_2||g||_2+4q^{\frac{d-3}{2}}\left(||F||_2||G||_2+||F'||_2||G'||_2\right)
\end{align*}
and in the case $d=2$,
\begin{align*}
\left|\mathcal{T}_t^*(f,g)-\frac{1}{q^2}||f||_1||g||_1\right|
\leq 3q^{-3}||f||_1||g||_1+4q^{d-1}||f||_2||g||_2+4q^{\frac{d-3}{2}}\left(||F||_2||G||_2+||F'||_2||G'||_2\right)
\end{align*}
\end{corollary}
In order to prove Theorem \ref{main}, we will consider the case when $f(x,y)=P_k^*(x,y)$, $g(z,w)=P_{\ell}^*(z,w)$ for some $k,\ell$.  
\begin{lemma}\label{calculation}
Let $f(x,y)=P_k^*(x,y)$ with $k\geq 2$.  If $|E|\geq \frac{2k}{\log{2}}q^{\frac{d+1}{2}}$, then
\begin{align*}
||f||_{1}&=\mathcal{P}_k^*=\frac{|E|^{k+1}}{q^k}+\mathcal{E}_k \nonumber \\
||f||_{2}&=\sqrt{C_{2k}^*} \nonumber \\
||F||_{2}&=||F'||_{2}=\sqrt{\mathcal{P}_{2k}^*}\leq 2\frac{|E|^{k+\frac{1}{2}}}{q^k}
\end{align*}
Where
\begin{align*}
|\mathcal{E}_k|\leq \frac{2k}{\log{2}}q^{\frac{d+1}{2}}\frac{|E|^k}{q^k}
\end{align*}
\end{lemma}
\begin{proof}
Unravel definitions, apply Theorem \ref{Chains} in the dot-product case and the analogous result from \cite{BCCHIP} in the distance case, and a degree 1 Taylor approximation of $\sqrt{x}$.
\end{proof}
\begin{lemma}\label{calculation2}
Let $\ell\geq k$.  If $|E|\geq \frac{2\ell}{\log{2}} q^{\frac{d+1}{2}}$, then
\begin{align*}
\left|\frac{||P_k^*||_1\cdot ||P_{\ell}^*||_1}{q^2}-\frac{|E|^{k+\ell+2}}{q^{k+\ell+2}}\right|
\leq (8k+4\ell)q^{\frac{d+1}{2}}\frac{|E|^{k+\ell+1}}{q^{k+\ell+2}}
\end{align*}
\end{lemma}
\begin{proof}
It follows directly from Theorem \ref{Chains} in the dot-product case and the analogous result from \cite{BCCHIP} in the distance case that
\begin{align*}
\left|\frac{||P_k^*||_1\cdot ||P_{\ell}^*||_1}{q^2}-\frac{|E|^{k+\ell+2}}{q^{k+\ell+2}}\right|\leq 4(k+\ell)q^{\frac{d+1}{2}}\frac{|E|^{k+\ell+1}}{q^{k+\ell+2}}+16k\ell q^{d-1}\frac{|E|^{k+\ell}}{q^{k+\ell}}
\end{align*}
Using the fact that $q^{\frac{d+1}{2}}\leq \frac{1}{4\ell}|E|$ finishes the proof of the lemma.  
\end{proof}
We are now ready to prove the main theorem.
\begin{proof}[Proof of Theorem \ref{main}]
We proceed by induction.  Using Corollary \ref{concise} and Lemma \ref{calculation} and Lemma \ref{calculation2}, as well as the estimates from \cite{HI08} and \cite{IR07} for the number of edges in $\mathcal{G}_t^*$, we start by computing $C_4^*$ and $C_5^*$.  In order to handle the $d=2$ case and the $d\geq 3$ case at once, let
\begin{align*}
\gamma=\left\{
\begin{array}{ll}
-1 & :d=2 \\
-\frac{d-2}{2} & :d\geq 3
\end{array}
\right.
\end{align*}
Then we get
\begin{align*}
\left|C_4^*-\frac{|E|^4}{q^4}\right|
&=\left|\mathcal{T}_t^*(P_1^*,P_1^*)-\frac{|E|^4}{q^4}\right|
\leq \left|\mathcal{T}_t^*(P_1^*,P_1^*)-\frac{||P_1^*||_1^2}{q^2}\right|+\left|\frac{||P_1^*||_1^2}{q^2}-\frac{|E|^4}{q^4}\right| \\
&\leq 3q^{\gamma-2}{\mathcal{P}_1^*}^2+4q^{d-1}\mathcal{P}_1^*+8q^{\frac{d-3}{2}}\mathcal{P}_2^*+12q^{\frac{d-7}{2}}|E|^3 \\
&\leq \frac{|E|^4}{q^4}\left(12q^{\gamma}+8\frac{q^{d+2}}{|E|^2}+28\frac{q^{\frac{d+1}{2}}}{|E|}\right)
\end{align*}
In particular, under the hypotheses of the theorem we have $C_4^*\leq 2\frac{|E|^4}{q^4}$.  Therefore,
\begin{align*}
\left|C_5^*-\frac{|E|^5}{q^5}\right|
&=\left|\mathcal{T}_t^*(P_1^*,P_2^*)-\frac{|E|^5}{q^5}\right|
\leq \left|\mathcal{T}_t^*(P_1^*,P_2^*)-\frac{||P_1^*||_1||P_2^*||_1}{q^2}\right|+\left|\frac{||P_1^*||_1||P_2^*||_1}{q^2}-\frac{|E|^5}{q^5}\right| \\
&\leq 3q^{\gamma-2}\mathcal{P}_1^*\mathcal{P}_2^*+4q^{d-1}\sqrt{\mathcal{P}_1^*C_4^*}+8q^{\frac{d-3}{2}}\sqrt{\mathcal{P}_2^*\mathcal{P}_4^*}+16q^{\frac{d-9}{2}}|E|^4 \\
&\leq \frac{|E|^5}{q^5}\left(12q^{\gamma}+8\frac{q^{\frac{2d+3}{2}}}{|E|^2}+32\frac{q^{\frac{d+1}{2}}}{|E|}\right)
\end{align*}
With an estimate for $C_4^*$, we can handle any even length inductively:
\begin{align*}
\left|C_{2k+2}^*-\frac{|E|^{2k+2}}{q^{2k+2}}\right|
&=\left|\mathcal{T}_t^*(P_k^*,P_k^*)-\frac{|E|^{2k+2}}{q^{2k+2}}\right| \\
&\leq \left|\mathcal{T}_t^*(P_k^*,P_k^*)-\frac{||P_k^*||_1||P_k^*||_1}{q^2}\right|+\left|\frac{||P_k^*||_1||P_k^*||_1}{q^2}-\frac{|E|^{2k+2}}{q^{2k+2}}\right| \\
&\leq 3q^{\gamma-2}{\mathcal{P}_k^*}^2+4q^{d-1}C_{2k}^*+8q^{\frac{d-3}{2}}\mathcal{P}_{2k}^*+12kq^{\frac{d+1}{2}}\frac{|E|^{2k+1}}{q^{2k+2}} \\
&\leq \frac{|E|^{2k+2}}{q^{2k+2}}\left(12q^{\gamma}+8\frac{q^{d+1}}{|E|^2}+(16+12k)\frac{q^{\frac{d+1}{2}}}{|E|}\right)
\end{align*}
Finally, noting that under the hypotheses of the theorem we can conclude that $C_n^*\leq 2\frac{|E|^n}{q^n}$ for any even $n\geq 4$, we can handle the odd case:
\begin{align*}
\left|C_{2k+1}^*-\frac{|E|^{2k+1}}{q^{2k+1}}\right|
&=\left|\mathcal{T}_t^*(P_k^*,P_{k-1}^*)-\frac{|E|^{2k+1}}{q^{2k+1}}\right| \\
&\leq \left|\mathcal{T}_t^*(P_k^*,P_{k-1}^*)-\frac{||P_k^*||_1||P_{k-1}^*||_1}{q^2}\right|+\left|\frac{||P_k^*||_1||P_{k-1}^*||_1}{q^2}-\frac{|E|^{2k+1}}{q^{2k+1}}\right| \\
&\leq 3q^{\gamma-2}\mathcal{P}_k^*\mathcal{P}_{k-1}^*+4q^{d-1}\sqrt{C_{2k}^*C_{2k-2}^*}+8q^{\frac{d-3}{2}}\sqrt{\mathcal{P}_{2k}^*\mathcal{P}_{2k-2}^*}+(12k-8)q^{\frac{d+1}{2}}\frac{|E|^{2k+1}}{q^{2k+2}} \\
&\leq \frac{|E|^{2k+1}}{q^{2k+1}}\left(12q^{\gamma}+8\frac{q^{d+1}}{|E|^2}+(24+12k)\frac{q^{\frac{d+1}{2}}}{|E|}\right)
\end{align*}
\end{proof}
The proof for Theorem \ref{main2} is quite similar, but we will need the following lemma.
\begin{lemma}\label{bound}
Let $\psi_k(\alpha)=(k-1)\alpha-k+2$.  Then whenever $\frac{k-2}{k-1}\leq \alpha<1$ and $|E|\geq q^{\frac{d+2-\alpha}{2}}$, for sufficiently large $q$ we have
\begin{align}
C_{2k}^*\leq A_k\frac{|E|^{2k}}{q^{2k}}q^{\psi_k(\alpha)}
\end{align}
Where $A_k=10\cdot 6^{k-2}$
\end{lemma}
\begin{proof}
For $k=2$, if $|E|\geq q^{\frac{d+2-\alpha}{2}}$ then for sufficiently large $q$,
\begin{align*}
C_4^*
&\leq \frac{|E|^4}{q^4}\left(1+12q^{\gamma}+8\frac{q^{d+2}}{|E|^2}+28\frac{q^{\frac{d+1}{2}}}{|E|}\right) \\
&\leq \frac{|E|^4}{q^4}\left(1+12q^{\gamma}+8q^{\alpha}+28q^{\frac{\alpha}{2}-\frac{1}{2}}\right) 
\leq 10\frac{|E|^4}{q^4}q^{\alpha},
\end{align*}
and so the result holds for $A_2=10$.  Assuming it holds for $C_{2k}^*$, we find for sufficiently large $q$ that
\begin{align*}
C_{2k+2}^*
&\leq \frac{|E|^{2k+2}}{q^{2k+2}}+3q^{\gamma-2}{\mathcal{P}_k^*}^2+4q^{d-1}C_{2k}^*+8q^{\frac{d-3}{2}}\mathcal{P}_{2k}^*+12kq^{\frac{d+1}{2}}\frac{|E|^{2k+1}}{q^{2k+2}} \\
&\leq \frac{|E|^{2k+2}}{q^{2k+2}}\left(1+12q^{\gamma}+4A_k\frac{q^{d+1+\psi_k(\alpha)}}{|E|^2}+(16+12k)\frac{q^{\frac{d+1}{2}}}{|E|}\right) \\
&\leq \frac{|E|^{2k+2}}{q^{2k+2}}\left(1+12q^{\gamma}+4A_kq^{\psi_k(\alpha)+\alpha-1}+(16+12k)q^{\frac{\alpha}{2}-\frac{1}{2}}\right) \\
&\leq 6A_k\frac{|E|^{2k+2}}{q^{2k+2}}q^{\psi_{k+1}(\alpha)}
\end{align*}
and thus it holds for $C_{2k+2}$ as long as $\psi_{k+1}(\alpha)\geq 0$, which is the case when $\alpha\geq \frac{k-1}{k}$. 
\end{proof}
\begin{proof}[Proof of Theorem \ref{main2}]
The argument is similar to the proof of Theorem \ref{main}.  From the estimate for $C_4^*$, we see that if $|E|\geq q^{\frac{d+2-\alpha}{2}}$ for $0\leq \alpha<1$, then $C_4^*\leq A_2\frac{|E|^4}{q^4}q^{\alpha}$.  In that case,
\begin{align*}
\left|C_5^*-\frac{|E|^5}{q^5}\right|
&\leq 3q^{\gamma-2}\mathcal{P}_1^*\mathcal{P}_2^*+4q^{d-1}\sqrt{\mathcal{P}_1^*C_4^*}+8q^{\frac{d-3}{2}}\sqrt{\mathcal{P}_2^*\mathcal{P}_4^*}+16q^{\frac{d-9}{2}}|E|^4 \\
&\leq \frac{|E|^5}{q^5}\left(12q^{\gamma}+8A_k\frac{q^{\frac{2d+3+\alpha}{2}}}{|E|^2}+32\frac{q^{\frac{d+1}{2}}}{|E|}\right)
\end{align*}
In particular, if $\alpha=\frac{1}{3}-\delta$ for $0<\delta<\frac{2}{9}$, then \begin{align*}
\left|C_5^*-\frac{|E|^5}{q^5}\right|
&\leq \frac{|E|^5}{q^5}\left(12q^{\gamma}+8A_kq^{-\frac{3}{2}\delta}+32q^{-\frac{\delta}{2}-\frac{1}{3}}\right) \\
&\leq (44+8A_k)\frac{|E|^5}{q^5}q^{-\frac{3}{2}\delta}
\end{align*}
  For the even case, we use Lemma \ref{bound} to find that
\begin{align*}
\left|C_{2k+2}^*-\frac{|E|^{2k+2}}{q^{2k+2}}\right|
&\leq 3q^{\gamma-2}{\mathcal{P}_k^*}^2+4q^{d-1}C_{2k}^*+8q^{\frac{d-3}{2}}\mathcal{P}_{2k}^*+12kq^{\frac{d+1}{2}}\frac{|E|^{2k+1}}{q^{2k+2}} \\
&\leq \frac{|E|^{2k+2}}{q^{2k+2}}\left(12q^{\gamma}+8A_k\frac{q^{d+1+\psi_k(\alpha)}}{|E|^2}+(16+12k)\frac{q^{\frac{d+1}{2}}}{|E|}\right)
\end{align*}
and in particular, if $\alpha=\frac{k-1}{k}-\delta$ for $0<\delta<\frac{1}{2k^2}$ then 
\begin{align*}
\left|C_{2k+2}^*-\frac{|E|^{2k+2}}{q^{2k+2}}\right|
&\leq \frac{|E|^{2k+2}}{q^{2k+2}}\left(12q^{\gamma}+8A_kq^{-k\delta}+(16+12k)q^{-\frac{1}{2k}-\frac{\delta}{2}}\right) \\
&\leq (28+8A_k+12k)\frac{|E|^{2k+2}}{q^{2k+2}}q^{-k\delta}
\end{align*}
And finally the odd case; if $\alpha\geq \frac{k-2}{k-1}$, then
\begin{align*}
\left|C_{2k+1}^*-\frac{|E|^{2k+1}}{q^{2k+1}}\right|
&\leq 3q^{\gamma-2}\mathcal{P}_k^*\mathcal{P}_{k-1}^*+4q^{d-1}\sqrt{C_{2k}^*C_{2k-2}^*}+8q^{\frac{d-3}{2}}\sqrt{\mathcal{P}_{2k}^*\mathcal{P}_{2k-2}^*}+(12k-8)q^{\frac{d+1}{2}}\frac{|E|^{2k+1}}{q^{2k+2}} \\
&\leq \frac{|E|^{2k+1}}{q^{2k+1}}\left(12q^{\gamma}+8A_k\frac{q^{d+1+\frac{1}{2}((2k-3)\alpha-2k+5)}}{|E|^2}+(24+12k)\frac{q^{\frac{d+1}{2}}}{|E|}\right)
\end{align*}
Note that $\frac{2k-3}{2k-1}\geq \frac{k-2}{k-1}$ for all $k$.  Thus, if $\alpha=\frac{2k-3}{2k-1}-\delta$ for 
\begin{align*}
0<\delta\leq \frac{1}{(k-1)(2k-1)}
\end{align*}
then
\begin{align*}
\left|C_{2k+1}^*-\frac{|E|^{2k+1}}{q^{2k+1}}\right|
&\leq \frac{|E|^{2k+1}}{q^{2k+1}}\left(12q^{\gamma}+8A_kq^{-\left(k-\frac{1}{2}\right)\delta}+(24+12k)q^{-\frac{2}{2k-1}-\frac{\delta}{2}}\right) \\
&\leq (36+8A_k+12k)\frac{|E|^{2k+1}}{q^{2k+1}}q^{-\left(k-\frac{1}{2}\right)\delta}
\end{align*}

\end{proof}
\section{Existence of Non-Degenerate Cycles}\label{degenerate}

\vskip.125in

The following bound follows directly from Theorem \ref{main} in the case $n=4$, and Theorem \ref{main2} for $n\geq 5$:
\begin{corollary}\label{mainthmcor}
For $n\geq 4$ and $q$ sufficiently large, if
\begin{align*}
|E|\geq \left\{\begin{array}{ll}
q^{\frac{1}{2}(d+2-\frac{k-2}{k-1}+\delta)} & :n=2k \\
q^{\frac{1}{2}(d+2-\frac{2k-3}{2k-1}+\delta)} & :n=2k+1
\end{array}
\right.
\end{align*}
then
\begin{align*}
\left|C_n^*-\frac{|E|^n}{q^n}\right|\leq 
K_n\frac{|E|^n}{q^n}q^{-\left(\frac{n}{2}-1\right)\delta}
\end{align*}
where $K_4=48$, and for $n\geq 5$ $K_n=36+80\cdot 6^{\left\lfloor \frac{n}{2}\right\rfloor-2}+12\left\lfloor\frac{n}{2}\right\rfloor$.
\end{corollary}
We will show that $\mathcal{G}_t^*$ admits $\sim\frac{|E|^k}{q^k}$ non-degenerate $k$-cycles under the hypotheses of Corollary \ref{mainthmcor}.  By non-degenerate, we mean a cycle in $\mathcal{G}_t^*$ with $k$ distinct vertices and edges.  Our estimates for $C_k^*$ do not on their own imply the existence of even one non-degenerate cycle, because our estimates include all the degenerate cycles as well as non-degenerate.  In order to prove Theorem \ref{nondegenerate}, we reduce to a subset of $E$ which we can show admits $\sim\frac{|E|^k}{q^k}$ non-degenerate $k$-cycles, which implies that $E$ also admits $\sim\frac{|E|^k}{q^k}$ non-degenerate $k$-cycles.  
\begin{definition}
For $\lambda=\lambda(q)>0$ to be specified later (with the property that $\lambda(q)\to\infty$), let
\begin{align*}
E^{dist}&=\left\{x\in E: \ E\ast S_t(x)\leq \lambda \frac{|E|}{q}\right\}, \nonumber \\
E^{prod}&=\left\{x\in E: \ \sum_y{E(y)D_t(x,y)}\leq \lambda \frac{|E|}{q}\right\}
\end{align*}
As before, we use $E^*$ to mean $E^{dist}$ resp. $E^{prod}$.  
\end{definition}
Then by Chebyshev's inequality $|E\setminus E^*|\leq 2\frac{|E|}{\lambda}$ for $q$ sufficiently large, since
\begin{align*}
|E\setminus E^{dist}|\leq \frac{q}{\lambda|E|}\sum_x{E\ast S_t(x)}\leq 2\frac{|E|}{\lambda}, \nonumber \\
|E\setminus E^{prod}|\leq \frac{q}{\lambda|E|}\sum_{x,y}{E(x)E(y)D_t(x,y)}\leq 2\frac{|E|}{\lambda}
\end{align*}
Note that any degenerate cycle is a connected graph on at most $k-1$ vertices, so it has a spanning tree which has the same number of vertices.  
\begin{definition}
For a tree $T$ with vertex set $V$ and $|V|\leq n-1$, let $n_T^*$ be the number of edge-preserving maps of $T$ into $\mathcal{G}_t^*(E^*)$, and let $G_T$ be the number of graphs on $V$ with at most $n-1$ vertices for which $T$ is a spanning tree.  
\end{definition}
Then the number of degenerate $n$-cycles in $\mathcal{G}_t^*(E^*)$ is bounded by $\sum_T{n_T^*G_T}$, where the sum is taken over trees with at most $n-1$ vertices.  
\begin{lemma}\label{tree lemma}
If $T$ has $r+1$ vertices and hence $r$ edges, then 
\begin{align*}
\left|n_T^*-\frac{|E|^{r+1}}{q^r}\right|
\leq 4r\left(\lambda^{-1}+\lambda^{\frac{r-1}{2}}\frac{q^{\frac{d+1}{2}}}{|E|}\right)
\end{align*}
\end{lemma}
\begin{proof}[Proof of Theorem \ref{tree counting}]
Setting $\lambda=q^{\frac{2\varepsilon}{r+1}}$ in Lemma \ref{tree lemma} yields Theorem \ref{tree counting} as a corollary.
\end{proof}
\begin{proof}[Proof of Lemma \ref{tree lemma}]
We proceed by induction on $r$, noting for the case $r=1$ that 
\begin{align*}
\left|n_T^*-\frac{|E^*|^2}{q}\right|\leq 2q^{\frac{d-1}{2}}|E^*|
\end{align*}
Moreover $|E|-|E^*|\leq 2\frac{|E|}{\lambda}$, so 
\begin{align*}
\left|n_T^*-\frac{|E|^2}{q}\right|&\leq \left|n_T^*-\frac{|E^*|^2}{q}\right|+\frac{1}{q}\left||E|^2-|E^*|^2\right| \\
&\leq 2q^{\frac{d-1}{2}}|E|+\frac{4|E|^2}{q\lambda}
\end{align*}
Fix a degree 1 vertex $v$ of $T$ with unique neighbor $w$, and let $T'$ be the tree obtained from $T$ by deleting $v$, and the edge from $v$ to $w$.  We will assume the result holds for $T'$, and deduce that it holds for $T$.  For $x\in E^*$, let $T'(x)$ be the number of edge-preserving maps $\sigma:T'\to \mathcal{G}_t^*(E^*)$ with $\sigma(w)=x$.  Similarly to how we have defined $E^*$ and $G_t^*$, let $D_t^*(x,y)=D_t(x,y)$ in the dot-product case, and $D_t^*(x,y)=S_t(x-y)$ in the distance case.  Then we have
\begin{align*}
n_T^*&=\sum_{x,y\in \mathbb{F}_q^d}{E^*(y)D_t^*(x,y)T'(x)}
=q^{-1}||T'||_{1} |E^*|+R, \ \ \text{where} \nonumber \\
R&\leq 2q^{\frac{d-1}{2}}||T'||_{2}|E^*|^{1/2}
\end{align*}
By our inductive assumption, 
\begin{align*}
\left|||T'||_1-\frac{|E|^r}{q^{r-1}}\right|
\leq (4r-4)\frac{|E|^{r}}{q^{r-1}}\left(\lambda^{-1}+\lambda^{\frac{r-2}{2}}\frac{q^{\frac{d+1}{2}}}{|E|}\right)
\end{align*}
and in particular for $q$ sufficiently large we have $||T'||_1\leq 2\frac{|E|^r}{q^{r-1}}$.  Thus,
\begin{align*}
\left|q^{-1}||T'||_1|E^*|-\frac{|E|^{r+1}}{q^r}\right|
&\leq q^{-1}||T'||_1(|E|-|E^*|)+\left|q^{-1}||T'||_1|E|-\frac{|E|^{r+1}}{q^r}\right| \\
&\leq \frac{|E|^{r+1}}{q^r}\left((4r-2)\lambda^{-1}+(4r-4)\lambda^{\frac{r-2}{2}}\frac{q^{\frac{d+1}{2}}}{|E|}\right)
\end{align*}
So to prove the lemma, it only remains to bound $R$, i.e. to bound $||T'||_{2}$.  Let $D$ be the maximum degree of a vertex in $\mathcal{G}_t^*(E^*)$, noting that by the definition of $E^*$,
\begin{align*}
D=\max_{x\in E^*}\sum_y{E^*(y)D_t^*(x,y)}\leq \lambda\frac{|E|}{q}
\end{align*}
Also, trivially for any $x\in E^*$, $T'(x)\leq D^{r-1}$.  Therefore,
\begin{align*}
||T'||_{2}^2\leq \max_x{T'(x)}\cdot ||T'(x)||_{1}
\leq 2\left(\frac{\lambda|E|}{q}\right)^{r-1}\frac{|E|^r}{q^{r-1}}
=2\lambda^{r-1}\frac{|E|^{2r-1}}{q^{2r-2}}
\end{align*}
and thus
\begin{align*}
R\leq 4\lambda^{\frac{r-1}{2}}q^{\frac{d+1}{2}}\frac{|E|^{r}}{q^r},
\end{align*}
We conclude that
\begin{align*}
\left|n_T^*-\frac{|E|^{r+1}}{q^r}\right|
\leq 4r\frac{|E|^{r+1}}{q^r}\left(\lambda^{-1}+\lambda^{\frac{r-1}{2}}\frac{q^{\frac{d+1}{2}}}{|E|}\right)
\end{align*}
completing the proof of the lemma.
\end{proof}
\begin{corollary}
With the notation from Theorem \ref{tree counting}, if $|E|\geq q^{\frac{d+1}{2}+\varepsilon}$ then for sufficiently large $q$,
\begin{align*}
n_T^*
\leq 2\frac{|E|^{r+1}}{q^r}
\end{align*}
\end{corollary}
\begin{proof}[Proof of Theorem \ref{nondegenerate}]
Let $E^*$ be defined as in Lemma \ref{tree lemma}, with $\lambda=q^{-\frac{2\varepsilon}{k-1}}$, $|E|\geq q^{\frac{d+1}{2}+\varepsilon}$, and
\begin{align*}
\varepsilon=\left\{\begin{array}{ll}
1-\frac{k-2}{k-1}+\delta & : n=2k \\
1-\frac{2k-3}{2k-1}+\delta & :n=2k+1
\end{array}
\right.
\end{align*}
Recall that the number of degenerate $n$-cycles in $\mathcal{G}_t^*(E^*)$ is bounded by $\sum_T{n_T^*G_T}$, where $G_T$ is the number of graphs on the same vertex set as $T$ for which $T$ is a spanning tree, and the sum is taken over trees with at most $n-1$ vertices.  Now, $T$ has $r$ edges while the $K_{r+1}$ has $\binom{r+1}{2}$ edges.  Since any subset of the edges in $K_{r+1}$ which are not in $T$ may be added to $T$ to form a graph whose spanning tree is $T$, we find that 
\begin{align*}
G_T=2^{\binom{r+1}{2}-r}
\end{align*}
Furthermore, it was shown by Cayley in 1889 that the number of trees on $r+1$ vertices is $(r+1)^{r-1}$.  Therefore, where $\mathcal{D}_k^*(E^*)$ is the number of degenerate cycles of length $n$ in $\mathcal{G}_t^*(E^*)$ and $\mathcal{N}_k^*(E^*)$ is the number of such non-degenerate cycles,
\begin{align*}
\mathcal{D}_n^*(E^*)
&\leq \sum_T{n_T^*G_T}
=\sum_{r=1}^{n-2}\sum_{|V|=r+1}{n_T^*G_T}
\leq c_n\sum_{r=1}^{n-2}{\frac{|E|^{r+1}}{q^r}} \nonumber \\
&=c_n\frac{|E|^2}{q}\cdot \frac{\frac{|E|^{n-2}}{q^{n-2}}-1}{\frac{|E|}{q}-1}\leq c_n\frac{|E|^{n-1}}{q^{n-2}},
\end{align*}
where $c_n=(n-1)^{n-3}\cdot 2^{\binom{n-1}{2}-n+3}$.  Therfore,
\begin{align*}
\left|\mathcal{N}_n^*(E^*)-\frac{|E|^n}{q^n}\right|
&\leq 
\left|C_n^*(E^*)-\frac{|E^*|^n}{q^n}\right|+q^{-n}(|E|^n-|E^*|^n)+\mathcal{D}_n^*(E^*) \\
&\leq K_n\frac{|E|^n}{q^n}q^{-\left(\frac{n}{2}-1\right)\delta}+nq^{-k}(|E|-|E^*|)|E|^{n-1}+c_n\frac{|E|^{n-1}}{q^{n-2}} \\
&\leq \frac{|E|^n}{q^n}\left(K_nq^{-\left(\frac{n}{2}-1\right)\delta}+2nq^{-\frac{2}{n-1}}+c_nq^{-\frac{d-3}{2}-\varepsilon}\right)
\end{align*}
Now, since 
\begin{align*}
\frac{|E|^{n}}{q^n}\left(1-\left(K_nq^{-\left(\frac{n}{2}-1\right)\delta}+2nq^{-\frac{2}{n-1}}+c_nq^{-\frac{d-3}{2}-\varepsilon}\right)\right) \leq  \mathcal{N}_n^*(E^*)\leq \mathcal{N}_n^*(E)\leq C_n^*(E)\leq \frac{|E|^{n}}{q^n}\left(1+K_nq^{-\left(\frac{n}{2}-1\right)\delta}\right)
\end{align*}
We obtain the desired bounds for $N_n^*(E)$, and establish the existence of non-degenerate $n$-cycles when $d\geq 3$.
\end{proof}

\end{document}